\documentclass{amsart}
\usepackage{graphicx}
\usepackage[mathscr]{eucal}
\usepackage{amssymb}
\newtheorem{thm}{Theorem}[section]
\newtheorem{cor}[thm]{Corollary}
\newtheorem{lem}[thm]{Lemma}
\newtheorem{prop}[thm]{Proposition}
\theoremstyle{definition}

\theoremstyle{remark}

\newtheorem{ex}[thm]{Example}
\newcommand{\norm}[1]{\left\Vert#1\right\Vert}

\newcommand{\set}[1]{\left\{#1\right\}}

\newcommand{\Pm}{\mathrm{Prim}}
\newcommand{\mset}{\emptyset}

\newcommand{\Id}{\mathrm{Id}}
\newcommand{\MP}{\textrm{Min-Primal}}

\newcommand{\cd}{$\cdot$}
\newcommand{\ci}{\circle*{1}}
\newcommand{\G}{\textrm{Glimm}}
\newcommand{\Sc}{\mathcal{S}}
\newcommand{\KH}{\mathcal{K}(H)}
\begin{document}
\baselineskip=18pt
\title{Centers of $C^*$-algebras rich in modular ideals}
\author{Aldo J. Lazar}
\address{School of Mathematical Sciences\\
         Tel Aviv University\\
         Tel Aviv 69778, Israel}
\email{aldo@post.tau.ac.il}
\date{April 3, 2011}

\thanks{}%

\subjclass{46L05}

\keywords{$C^*$-algebra, primitive ideal, modular ideal, $C_0(X)$-algebra}

%\commby{}%
% ---------------------------------------------------------------
\begin{abstract}

We provide, in the spirit of \cite{De}, new conditions under which a $C^*$-algebra has a nonzero center. We also present an example of a
separable AF algebra with center $\{0\}$ but whose all the primitive ideals are modular, thus answering a question from \cite{De}.

\end{abstract}
\maketitle
% ------------------------------------------------------------
\section{Introduction}

Clearly every primitive ideal of a $C^*$-algebra $A$ that does not contain the center of $A$ is modular. It is also obvious that the set of all
these ideals is open in $\Pm(A)$. Thus, if the center of $A$ is nonzero, the set of its modular primitive ideals has a nonempty interior in
$\Pm(A)$. The main purpose of \cite{De} is an investigation of the converse: does the existence of a nonempty open set of modular primitive
ideals imply a nonzero center? Among other results, an affirmative answer is obtained for liminal $C^*$-algebras. However, two examples of
postliminal $C^*$-algebras with zero center are given there: one separable which has a nonempty open set of modular primitive ideals and another
one that is nonseparable but whose all primitive ideals are modular.

Here we treat conditions which ensure that a $C^*$-algebra whose all its minimal primal ideals (the definition follows) are modular has a
nonzero center. In particular we treat the case of a postliminal algebra. In section \ref{second} we give an example of a postliminal AF algebra
with zero center whose all primitive ideals are modular. This answers a question of Delaroche, \cite[p. 126]{De}.

By the term ideal we shall mean everywhere a two sided closed ideal. $\Id(A)$ will denote the collection of all the ideals of the $C^*$-algebra
$A$. For $I\in \Id(A)$ we shall let $\theta_I : A\to A/I$ be the quotient map. On $Id(A)$ we shall consider a compact Hausdorff topology; a net
$\{I_{\alpha}\}$ converges to $I$ in this topology if and only if $\norm{\theta_{I_{\alpha}}(a)}\to \norm{\theta_I(a)}$ for every $a\in A$, see
\cite{A} for more on this topology. If it is not mentioned otherwise, $\Id(A)$ and its subsets will be endowed with this topology. However, on
the primitive ideal space of $A$, denoted $\Pm(A)$, we shall always work with the usual Jacobson topology. A primal ideal $I$ of a $C^*$-algebra
$A$ is defined by the following property: whenever $I_1, \ldots I_n$, $n\geq 2$, are ideals of $A$ such that $I_1\cdot I_2\cdots I_n = \{0\}$
then $I_k\subseteq I$ for some $k$. Every prime (in particular every primitive) ideal is primal and by using Zorn's lemma one sees that every
primal ideal must contain a minimal primal ideal. The collection of all the minimal primal ideals of $A$ is denoted by $\MP(A)$. See \cite{A}
and the references given there about primal ideals.

Two primitive ideals $P$, $Q$ of the $C^*$-algebra $A$ are said to be equivalent if $f(P) = f(Q)$ for every continuous $f : \Pm(A)\to
\mathbb{C}$. Each equivalence class is the hull of an ideal called a Glimm ideal of $A$; the collection of these ideals is denoted $\G(A)$ and
the quotient map $\phi_A : \Pm(A)\to \G(A)$ is called the complete regularization map, see \cite{AS}. $\G(A)$ will be considered with its
quotient topology induced by this map.

The following lemma is an immediate consequence of the Dauns-Hofmann theorem and the definition of the Glimm space so we omit its proof.

\begin{lem} \label{Glimm}

   Let $A$ be a $C^*$-algebra, $a\in A$ and $f : \emph{\G}(A)\to \mathbb{C}$ a bounded continuous function. Then there exists a unique $b\in A$
   such that $\theta_G(b) = f(G)\theta_G(a)$ for every $G\in \emph{\G}(A)$.

\end{lem}

An ideal $I$ is called semi-Glimm if it contains a Glimm ideal; this Glimm ideal is necessarily unique since its hull must contain the hull of
$I$. Obviously every Glimm ideal is semi-Glimm and every proper primal ideal is semi-Glimm by \cite[Lemma 2.2]{AS}. We set $S-\G(A)$ for the
family of all the semi-Glimm ideals of $A$ and we let $\psi_A$ be the map that takes each $I\in S-\G(A)$ to the Glimm ideal it contains.

\begin{lem} \label{cont}

The map $\psi_A : \emph{S}-\emph{\G}(A)\to \emph{\G}(A)$ is continuous.

\end{lem}

\begin{proof}

Let $\mathcal{U}$ be an open subset of $\G(A)$, $\mathcal{V} := \phi_A^{-1}(\mathcal{U})$ and denote by $J$ the ideal of $A$ for which $\Pm(J) =
\mathcal{V}$. Then for $I\in S-\G(A)$ we have $\psi_A(I)\in \mathcal{U}$ if and only if $\phi_A^{-1}(\psi_A(I))\in \mathcal{V}$ and this happens
if and only if the hull of $I$ is contained in $\mathcal{V}$. But for a semi-Glimm ideal $I$ this is equivalent to $J\nsubseteq I$. Now the set
$\{I\in \Id(A) \mid J\nsubseteq I\}$ is open in $\Id(A)$, see \cite[p. 525]{A}, and we are done.

\end{proof}

Observe that if $a,b,f$ are as in Lemma \ref{Glimm} and $I\in S-\G(A)$ then $\theta_I(b) = f(J)\theta_I(a)$, where $J := \psi_A(I)$, as follows
by using the canonical isomorphism of $A/I$ with $(A/J)/(I/J)$.

A family $\mathcal{F}$ of ideals of the $C^*$-algebra $A$ is called sufficiently large if \\ $\cup \{\Pm(A/I) \mid I\in \mathcal{F}\}$ is dense
in $\Pm(A)$.

For AF algebras we use the terminology of \cite{Da} and some which is self-explanatory but formalized in \cite{LT} like, for instance, the
notions of a level and a connected sequence in a Bratteli diagram. Recall that a subdiagram $E$ of a diagram $D$ of an AF algebra $A$ is the
diagram of an ideal $I$ of $A$ if and only if $E$ has the following two properties: the descendants of every vertex of $E$ belong to $E$ and if
every descendant of a vertex belongs to $E$ then that vertex itself belongs to $E$. If this is the case then $D\setminus E$ is a diagram of
$A/I$. The ideal $I$ is primitive if and only if every two vertices in $D\setminus E$ have a common descendant in $D\setminus E$, see
\cite[Theorem 3.8]{B}.

\section{Non trivial centers} \label{first}

Observe that if $a,b,f$ are as above and $I\in \mathrm{S-Glimm}(A)$ then $\theta_I(b) = f(J)\theta_I(a)$, where $J := \psi_A(I)$, as follows
from the canonical isomorphism of $A/I$ with $(A/J)/(I/J)$.

The proof of the following theorem is a variant of the proof of \cite[Theorem 3.7]{AEN}. Recall that $\mathrm{Glimm}(A)$ is considered with its
quotient topology which in the presence of a countable approximate identity is completely regular by \cite[Theorem 2.6]{L}.

\begin{thm} \label{main}

Let $A$ be a $C^*$-algebra that has a countable approximate identity and suppose there exists a sufficiently large Baire subspace $\mathcal{S}$
of $\emph{S-Glimm}(A)$ consisting of modular ideals. Suppose, moreover, that every non-void (relatively) open subset of $\mathcal{S}$ contains
the preimage by $\psi_A^{\mathcal{S}} := \psi_A|\Sc$ of a non-void relatively open subset of $\psi_A^{\Sc}(\Sc)\subseteq \emph{\G}(A)$. Then $A$
has a non-zero center.

\end{thm}

\begin{proof}

   With $I\in \Sc$ we shall denote by $\mathbf{1}_I$ the unit of $A/I$. If $A$ has a unit there is nothing to prove; otherwise let $\mathbf{1}$
   be the unit of $M(A)$, the multiplier algebra of $A$.
   Then $\tilde{A} := \{\lambda \mathbf{1} + a \mid \lambda\in \mathbb{C}, \; a\in A\}\subseteq M(A)$ is the $C^*$-algebra obtained by adjoining
   a unit to $A$. Letting $\tilde{\theta}_I : \tilde{A}\to A/I$ to be the obvious extension of $\theta_I, \; I\in \Sc,$ we have
   \begin{multline*}
      \norm{\lambda \mathbf{1}_I + \theta_I(a)} = \Vert\tilde{\theta}_I(\lambda \mathbf{1} + a)\Vert =
      \sup\{\Vert\tilde{\theta}_I(\lambda \mathbf{1} + a)\theta_I(b)\Vert \mid b\in A, \; \norm{b}\leq 1\} = \\
      \sup\{\norm{\theta_I(\lambda b + ab)} \mid b\in A, \; \norm{b}\leq 1\}.
   \end{multline*}
  We infer that the function $I\to \norm{\lambda \mathbf{1}_I + \theta_I(a)}$ is lower semi-continuous on $\Sc$. Thus
  $\{I\in \Sc \mid \norm{\lambda \mathbf{1}_I + \theta_I(a)}\leq \alpha\}$ is a closed subset of $\Sc$ for every $\alpha > 0$.

Let now $\{u_n\}$ be a positive countable approximate unit of $A$. Clearly $\lim_{n\to \infty} \theta_I(u_n) = ~\mathbf{1}_I$ for every $I\in
\Sc$ hence
\[
 \cup_{n=1}^{\infty} \{I\in \Sc \mid \norm{\mathbf{1}_I - \theta_I(u_n)}\leq 1/2\} = \Sc.
\]
Since $\Sc$ is a Baire space, there must be a natural number $n_0$ such that the closed set $\{I\in \Sc \mid \norm{\mathbf{1}_I -
\theta_I(u_{n_0})}\leq 1/2\}$ has a non-void interior $\mathcal{E}$.

If $I\in \mathcal{E}$ the spectrum of $\theta_I(u_{n_0})$ is included in the interval $[1/2,1]$. Choosing a continuous function $f : [0,1]\to
[0,1]$ such that $f(0) = 0$ and $f(t) =1$ for $t\in [1/2,1]$ we get an element $a := f(u_{n_0})\in A$ such that $\theta_I(a) = \mathbf{1}_I$ if
$I\in \mathcal{E}$. Let $\mathcal{U}$ be an open subset of $\G(A)$ such that $\mathcal{U}\cap \psi_A^{\Sc}(\Sc)\neq \mset$ and
${\psi_A^{\Sc}}^{-1}(\mathcal{U})\subseteq \mathcal{E}$, $J_0\in \mathcal{U}\cap \psi_A^{\Sc}(\Sc)$, and $g : \G(A)\to [0,1]$ a continuous
function that satisfies $g(J_0) = 1$ and $g(J) = 0$ for $J\notin \mathcal{U}$. With $b\in A$ given by Lemma \ref{Glimm} for $a$ and $g$, i.e.
$\theta_J(b) = g(J)\theta_J(a)$ for every $J\in \G(A)$, we have $\theta_I(b) = 0$ if $I\in \Sc\setminus {\psi_A^{\Sc}}^{-1}(\mathcal{U})$.
Indeed, if $I\in \Sc\setminus {\psi_A^{\Sc}}^{-1}(\mathcal{U})$ then $g(\psi_A^{\Sc}(I)) = 0$. On the other hand, if $I\in
{\psi_A^{\Sc}}^{-1}(\mathcal{U})$ and $J := \psi_A^{Sc}(I)$ we have $\theta_I(b) = g(J)\mathbf{1}_I$. Consequently $\theta_I(bc - cb) = 0$ for
every $I\in \Sc$ and every $c\in A$.$\Sc$ is sufficiently large hence $b$ is in the center of $A$. Now for $I_0\in {\psi_A^{Sc}}^{-1}(J_0)$ one
has $\theta_{I_0}(b) = \theta_{I_0}(a) = \mathbf{1}_{I_0}\neq 0$ and we conclude that the center of $A$ is non-zero.

The following result can be also obtained as an easy consequence of \cite[Lemma 3.6]{AS1}.

\end{proof}

\begin{prop} \label{open}

   Let $A$ be a $C^*$-algebra with a countable approximate identity. Suppose that $\phi_A : \emph{Prim}(A)\to \emph{\G}(A)$ is open and each
   Glimm ideal is modular. Then $A$ has a non-zero center.

\end{prop}

\begin{proof}

   For every $a\in A$ and $\alpha > 0$ we have $\{G\in \G(A) \mid \norm{\theta_G(a)}\geq \alpha\} = \phi_A(\{P\in \Pm(A) \mid
   \norm{\theta_P(a)}\geq \alpha\})$ and $\{G\in \G(A) \mid \norm{\theta_G(a)} > \alpha\} = \phi_A(\{P\in \Pm(A) \mid \norm{\theta_P(a)} >
   \alpha\})$. The set $\{G\in \G(A) \mid \norm{\theta_G(a)}\geq \alpha\}$ is closed in the Hausdorff space $(\G(A),\tau_q)$ as the continuous
   image of a compact set and $\{G\in \G(A) \mid \norm{\theta_G(a)} > \alpha\}$ is open by our assumption on $\phi_A$. Thus $G\to \norm{\theta_G(a)}$
   is continuous on $(\G(A),\tau_q)$ and we conclude that the identity map from $(\G(A),\tau)$ to $(\G(A),\tau_q)$, which is the restriction of
   $\psi_A$ to $(\G(A),\tau)$, is a homeomorphism. From the fact that $\phi_A$ is open we also infer that $(\G(A),\tau_q)$ is a locally compact
   Hausdorff space hence a Baire space. We get the conclusion from Theorem \ref{main}.

\end{proof}

A topological space $X$ is called quasi-completely regular if for every non-void open subset $U$ of $X$ there is a non-zero real valued
continuous function on $X$ that is identically $0$ on $X\setminus U$. Such spaces were called "quasi-uniformisable" in \cite{De} but this term
is used in another sense in topology.

\begin{prop} \label{qcr}

    Let $A$ be a $C^*$-algebra that has a countable approximate identity. Suppose that every minimal primal ideal of $A$ is modular and $\Pm(A)$
    is quasi-completely regular. Then $A$ has a non-zero center.

\end{prop}

\begin{proof}

    Let $I$ be an ideal of $A$. An ideal of $A/I$ has the form $J/I$ with $J$ an ideal of $A$. It is immediately seen that if $I/J$ is a minimal
    primal ideal of $A/I$ then $J$ is a primal ideal of $A$. Thus $A/J$ has a unit and so does $(A/I)/A/J)$ which is isomorphic to $A/J$.
    Clearly $\Pm(I)$ as an open subset of $\Pm(A)$ is quasi-completely regular.

    Every primitive ideal of $A$ contains a minimal primal ideal so $\MP(A)$ is sufficiently large. It is a Baire space by \cite[Proposition
    4.9]{A}. The restriction $\varphi_A$ of $\psi_A$ to $\MP(A)$ maps this space onto $\G(A)$, again since every primitive ideal contains a minimal
    primal ideal. We are going to show now that every non-void open subset of $\MP(A)$ contains the preimage by $\varphi_A$ of an open subset
    of $\G(A)$. So let
    $\mathcal{U}$ be a non-void open subset of $\MP(A)$. By \cite[Corollary 4.3(a)]{A}, $\mathcal{U}$ is the union of sets of the form $\mathcal{V} :=
    \{I\in \MP(A) \mid I_j\nsubseteq I, \; 1\leq j\leq n\}$ where $\{I_j\}_{j=1}^n$ is a set of ideals of $A$. So let $\mathcal{V}\neq \mset$ be such
    a set contained in $\mathcal{U}$. Since $\mathcal{V}$ contains at least one minimal primal ideal we must have $J := \cap_{j=1}^n I_j\neq
    \mset$ and clearly $\{I\in \MP(A) \mid J\nsubseteq I\}\subseteq \mathcal{V}$. $\Pm(A)$ is quasi-completely regular so there exists a non-zero
    continuous function $f : \Pm(A)\to \mathbb{R}$ that vanishes off $\Pm(J)$. Let $g : \G(A)\to \mathbb{R}$ be such that $f = g\circ \phi_A$.
    Then $\{G\in \G(A) \mid g(G) > 0\}$ is open and its preimage by $\varphi_A$ is contained in $\mathcal{V}$. Indeed, suppose $I\in \MP(A)$ and
    $g(\varphi_A(I)) > 0$. With $P$ a primitive ideal that contains $I$ we have $P\supseteq I\supseteq \varphi_A(I)$ hence
    \[
     f(P) = g(\phi_A(P)) = g(\phi_A(I)) > 0.
    \]
    Thus $P\in \Pm(J)$ so $I$ cannot contain $J$ which means $I\in \mathcal{V}$. Theorem \ref{main} implies that $A$ has a non-zero center.

\end{proof}

The $C^*$-algebra obtained by adjoining a unit to the ideal of compact operators on an infinite-dimensional Hilbert space is an example that
satisfies the conditions of Proposition \ref{open} but not those of Proposition \ref{qcr}. In Section \ref{second} we shall give an example of a
$C^*$-algebra in the situation described by Proposition \ref{qcr} for which the complete regularization map is not open.

A $C^*$-algebra $A$ was called in \cite[D\'{e}finition 4]{De} generalized quasi-central if for every ideal $I$ of $A$, $I\neq A$, the center of
$A/I$ is non-zero.

\begin{cor} \label{gqc}

Let $A$ be a $C^*$-algebra that has a countable approximate identity. Suppose that every minimal primal ideal of $A$ is modular and every closed
subset of $\Pm(A)$ is a quasi-completely regular space with its relative topology. Then $A$ is generalized quasi-central

\end{cor}

\begin{proof}

 Let $I$ be an ideal of $A$. An ideal of $A/I$ has the form $J/I$ with $J$ an ideal of $A$. It is immediately seen that if $J/I$ is a minimal
    primal ideal of $A/I$ then $J$ is a primal ideal of $A$. Thus $A/J$ has a unit and so does $(A/I)/(J/I)$ which is isomorphic to $A/J$.
    $\Pm(A/I)$ as a closed subset of $\Pm(A)$ is quasi-completely regular. Thus the conclusion follows from Proposition \ref{qcr}.

\end{proof}

Obviously every quasi-completely regular space has the property that every non-empty open subset contains a closed subset with non-empty
interior. In certain topological spaces this easily verifiable property implies that the space is quasi-completely regular.

\begin{lem} \label{quasi}

   Let $X$ be a locally compact space that has an open dense Hausdorff subset. If every non-void subset of $X$ contains a closed subset with
   non-void interior then $X$ is quasi-completely regular.

\end{lem}

\begin{proof}

   Let $O$ be an open dense Hausdorff subset of $X$. If $U$ is any non-void open subset of $X$ then $U\cap O$ is a non void open set which is
   locally compact Hausdorff in its relative topology. Let $F$ be a closed subset of $U\cap O$ whose interior $V$ is non-void. $U\cap O$ is a
   completely regular space so  there is a non-zero real continuous function $f$ on $U\cap O$ which vanishes on $(U\cap O)\setminus V$. Now
   extend $f$ to all of $X$ by defining $f(x) := 0$ if $x\in X\setminus (U\cap O)$. Every point in $X\setminus (U\cap O)$ has a neighbourhood on
   which $f$ is identically zero, namely $X\setminus F$. Every point in $U\cap O$ has a neighbourhood on which $f$ is continuous, namely $U\cap
   O$ itself. Thus $f$ is continuous on $X$.

\end{proof}

\begin{prop} \label{post}

   Let $A$ be a postliminal $C^*$-algebra with a countable approximate identity. Suppose that every minimal primal ideal of $A$ is modular and
   $\Pm(A)$ has the property that every non-void open subset of $\Pm(A)$ contains a closed subset with non-empty interior. Then $A$ has a non-zero center.

\end{prop}

\begin{proof}

   $\Pm(A)$ contains an open dense Hausdorff subset by \cite[Theorem 4.5.5]{D}. Lemma \ref{quasi} and Proposition \ref{qcr} yield the
   conclusion.

\end{proof}

\section{Examples} \label{second}

The first example is a $C^*$-algebra $A$ that satisfies the conditions of Proposition \ref{post} for which $\phi_A$ is not open.

\begin{ex}

We adapt a construction from \cite[Example III.9.2]{DH}. We denote by $\KH$ the ideal of all the compact operators on a separable Hilbert space
$H$ and by $B$ the $C^*$-algebra generated by $\KH$ and the identity operator of $H$. $A$ is the $C^*$-algebra of all the continuous functions
$f : [-1,1]\to B$ such that $f(t)$ is diagonal with respect to a fixed orthonormal basis $\{e_n\}_{n=1}^{\infty}$ of $H$ whenever $0\leq t\leq
1$. Thus $f(t)$, $0\leq t\leq 1$, can be represented in the chosen basis by $\mathrm{diag}(f_1(t),f_2(t), \dots)$ where $f_n$ are scalar valued
continuous functions. Put $f_{\infty}(t) := \lim_{n\to \infty}f_n(t)$, $0\leq t\leq 1$; $f_{\infty}$ is a scalar valued continuous function too.

Clearly $A$ is a separable postliminal algebra. Its primitive ideals are: $P(t) := \{f\in A \mid f(t) = 0\}$, $Q(t) := \{f\in A \mid f(t)\in
\KH\}$ for $-1\leq t < 0$ and $R(t,n) := \{f\in A \mid f_n(t) = 0\}$ for $0\leq t\leq 1$, $1\leq n\leq \infty$. We are going now to list a
neighbourhood basis for each kind of kind of primitive ideal:
\begin{itemize}
   \item for $P(t_0)$, $-1\leq t_0 < 0$, the family of all the sets $\{P(t) \mid t\in (t_0-\eta,t_0+\eta)\cap [-1,0)\}$, with $\eta > 0$,
   \item for $Q(t_0)$, $-1\leq t_0 < 0$, the family of all the sets $\{P(t) \mid t\in (t_0-\eta,t_0+\eta)\cap [-1,0)\}\cup \{Q(t) \mid t\in
      (t_0-\eta,t_0+\eta)\cap [-1,0)\}$, with  $\eta > 0$,
   \item for $R(0,n_0)$, $1\leq n_0 < \infty$, the family of all the sets $\{P(t) \mid -\eta < t < 0\}\cup \{R(t,n_0) \mid 0\leq t < \eta\}$, with $0 <
      \eta <1$,
   \item for $R(0,\infty)$, the family of all the sets $\{P(t) \mid -\eta < t < 0\}\cup \{Q(t) \mid -\eta < t < 0\}\cup \{R(t,n) \mid 0\leq t < \eta, \;
      n_0 < n\leq \infty\}$, with $0 < \eta < 1$ and $1\leq n_0 < \infty$,
   \item for $R(t_0,n_0)$, $0 < t_0\leq 1$, $!\leq n_0 < \infty$, the family of all the sets $\{R(t,n_0) \mid t\in (t_0-\eta,t_0+\eta)\cap
   (0,1]\}$, with $\eta > 0$,
   \item for $R(t_0,\infty)$, $0 < t_0\leq 1$, the family of all the sets $\{R(t,n) \mid n > n_0, \; t\in (t_0-\eta,t_0+\eta)\cap (0,1]\}$, with
      $n_0\geq 1$ and $\eta > 0$.
\end{itemize}
It is easily seen that every non-empty open subset of $\Pm(A)$ contains a closed subset with non-empty interior.

The minimal primal ideals of $A$ are $P(t)$ for $1\leq t < 0$, $G := \cap \{R(0,n) \mid 1\leq n\leq\infty \}$, and $R(t,n)$ for $0 < t\leq 1$,
$1\leq n\leq \infty$ and every one is modular. These are also the Glimm ideals of $A$. Now, an open neighbourhood of $G$ in $\G(A)$ must contain
a set of the form $\{P(t) \mid -\eta < t <0\}\cup \{G\}\cup \{R(t,n) \mid 0 < t < \eta, \; 1\leq n\leq \infty\}$ for some $\eta\in (0,1)$ and
one easily sees that $\phi_A$ is not open. On the other hand, all the hypotheses of Proposition \ref{post} are fulfilled. We remark also that
$\Pm(A)$ is not quasi-separated so Proposition 14 of \cite{De} cannot be applied to $A$.

\end{ex}

As promised in the introduction we present now a postliminal (separable) AF algebra whose all primitive ideals are modular but with center
reduced to $\{0\}$. As a matter of fact all the minimal primal ideals of this algebra are modular so the hypothesis made on the primitive ideal
space in Proposition \ref{post} cannot be eliminated.

\begin{ex}

A Bratteli diagram of this algebra, $A$, appears in the figure that follows. In it the first vertex of the connected sequence $a_1$ should be
thought at the level 1 while the first vertex of the connected sequence $a_n$ should be imagined at the level $1 + 2(n - 1)$.

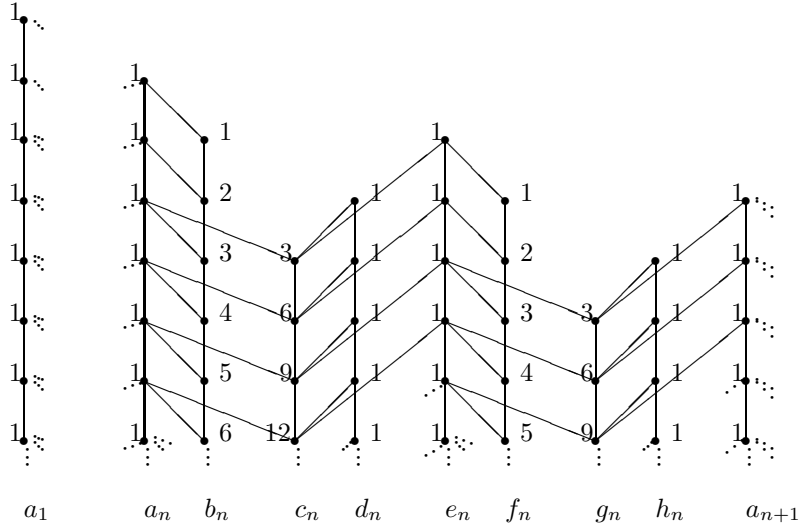
\setlength{\unitlength}{1mm}
\begin{picture}(200,100)
\put(0,10){$a_1$} \put(16,10){$a_n$} \put(24,10){$b_n$} \put(36,10){$c_n$} \put(44,10){$d_n$} \put(56,10){$e_n$} \put(64,10){$f_n$}
\put(76,10){$g_n$} \put(84,10){$h_n$} \put(96,10){$a_{n+1}$}

 \multiput(0,20)(0,8){7}{\line(0,1){8}} \multiput(0,20)(0,8){8}{\ci} \multiput(-2,20)(0,8){8}{$1$} \multiput(0,18)(0,-1){3}{\cd}
 \multiput(1,19)(0,8){8}{\cd} \multiput(1.5,18.5)(0,8){8}{\cd} \multiput(2,18)(0,8){8}{\cd} \multiput(1,19.7)(0,8){6}{\cd}
 \multiput(1.5,19.5)(0,8){6}{\cd} \multiput(2,19.3)(0,8){6}{\cd}

 \multiput(16,20)(0,8){6}{\line(0,1){8}} \multiput(16,20)(0,8){7}{\ci} \multiput(14,20)(0,8){7}{$1$} \multiput(15,18.8)(0,8){7}{\cd}
 \multiput(14,18.5)(0,8){7}{\cd} \multiput(13,18.2)(0,8){7}{\cd} \multiput(16,18)(0,-1){3}{\cd} \multiput(16,28)(0,8){6}{\line(1,-1){8}}
 \multiput(17,19)(0.5,-0.5){3}{\cd} \multiput(16,28)(0,8){4}{\line(5,-2){20}} \multiput(17,19.6)(1,-0.4){3}{\cd}

 \multiput(24,20)(0,8){5}{\line(0,1){8}} \multiput(24,20)(0,8){6}{\ci} \multiput(24,18)(0,-1){3}{\cd} \put(26,20){$6$} \put(26,28){$5$}
 \put(26,36){$4$} \put(26,44){$3$} \put(26,52){$2$} \put(26,60){$1$}

 \multiput(36,20)(0,8){3}{\line(0,1){8}} \multiput(36,20)(0,8){4}{\ci} \multiput(36,18)(0,-1){3}{\cd} \put(32,20){$12$} \put(34,28){$9$}
 \put(34,36){$6$} \put(34,44){$3$} \multiput(36,20)(0,8){4}{\line(1,1){8}} \multiput(36,20)(0,8){4}{\line(5,4){20}}

 \multiput(44,20)(0,8){4}{\line(0,1){8}} \multiput(44,20)(0,8){5}{\ci} \multiput(46,20)(0,8){5}{$1$} \multiput(44,18)(0,-1){3}{\cd}
 \multiput(43,19)(-0.5,-0.5){3}{\cd}

 \multiput(56,20)(0,8){5}{\line(0,1){8}} \multiput(56,20)(0,8){6}{\ci} \multiput(54,20)(0,8){6}{$1$} \multiput(56,28)(0,8){5}{\line(1,-1){8}}
 \multiput(56,28)(0,8){3}{\line(5,-2){20}} \multiput(55,18.4)(0,8){2}{\cd} \multiput(54,17.8)(0,8){2}{\cd} \multiput(53,17.2)(0,8){2}{\cd}
 \multiput(56,18)(0,-1){3}{\cd} \multiput(57,19)(0.5,-0.5){3}{\cd} \multiput(57,19.6)(1,-0.4){3}{\cd}

 \multiput(64,20)(0,8){4}{\line(0,1){8}} \multiput(64,20)(0,8){5}{\ci} \put(66,20){$5$} \put(66,28){$4$} \put(66,36){$3$} \put(66,44){$2$}
 \put(66,52){$1$} \multiput(64,18)(0,-1){3}{\cd}

 \multiput(76,20)(0,8){2}{\line(0,1){8}} \multiput(76,20)(0,8){3}{\ci} \put(74,20){$9$} \put(74,28){$6$} \put(74,36){$3$}
 \multiput(76,20)(0,8){3}{\line(1,1){8}} \multiput(76,20)(0,8){3}{\line(5,4){20}} \multiput(76,18)(0,-1){3}{\cd}

 \multiput(84,20)(0,8){3}{\line(0,1){8}} \multiput(84,20)(0,8){4}{\ci} \multiput(86,20)(0,8){4}{$1$} \multiput(84,18)(0,-1){3}{\cd}
 \multiput(83,19)(-0.5,-0.5){3}{\cd}

 \multiput(96,20)(0,8){4}{\line(0,1){8}} \multiput(96,20)(0,8){5}{\ci} \multiput(94,20)(0,8){5}{$1$} \multiput(95,19.2)(0,8){2}{\cd}
 \multiput(94,18.4)(0,8){2}{\cd} \multiput(93,17.6)(0,8){2}{\cd} \multiput(96,18)(0,-1){3}{\cd} \multiput(97,19)(0,8){5}{\cd}
 \multiput(98,18)(0,8){5}{\cd} \multiput(99,17)(0,8){5}{\cd} \multiput(97,19.4)(0,8){5}{\cd} \multiput(98,18.8)(0,8){5}{\cd}
 \multiput(99,18.2)(0,8){5}{\cd}

\end{picture}

It can be immediately checked that the diagram above has the property that for every connected sequence $\set{x_m}_{m=1}^{\infty}$ in it there
is a natural number $k$ such that for $m\geq k$ the vertex $x_{m+1}$ is a descendant of $x_m$ with multiplicity one. Hence, by \cite[Theorem
3.13]{LT}, $A$ is a postliminal algebra.

By direct examination one finds that the primitive quotients of $A$ have one of the following diagrams: $\{a_n\}$, $\{a_n,b_n\}$,
$\{a_n,c_n,d_n,e_n\}$, $\{d_n\}$, $\{e_n\}$, $\{e_n,f_n\}$, $\{e_n,g_n,h_n,a_{n+1}\}$, $\{h_n\}$, $n = 1,2, \ldots$ . We shall denote the
primitive ideals determined by the complementary diagrams by $P_n$, $Q_n$, $R_n$, $S_n$, $T_n$, $U_n$, $V_n$, $W_n$, $n = 1,2, \ldots$,
respectively. It is obvious that the diagrams of the primitive quotients are diagrams of unital AF algebras hence all the primitive ideals of
$A$ are modular. Moreover all the quotients of $A$ by its minimal primal ideals are modular. Indeed, by \cite[Theorem 2.1]{Be} an ideal $I$ of
an AF algebra is primal if and only if its associated diagram $D_I$ has the property that every finite set of vertices not in $D_I$ has a common
descendant in the diagram of the algebra. It is then easily seen that all the diagrams of the quotients of $A$ by the minimal primal ideals are
$\{a_n,b_n\}$, $\{a_n,c_n,d_n,e_n\}$, $\{e_n,f_n\}$, $\{e_n,g_n,h_n,a_{n+1}\}$, $n = 1,2, \ldots$ and all these quotients have units.

Now we are going to show that there are no nonzero elements in the center of $A$. To this end we prove that every real valued continuous
function on $\Pm(A)$ is constant. First remark that by the definition of the hull-kernel topology of the primitive ideal space we have:

\begin{multline*}
    \overline{\{P_n\}} = \{P_n\}, \overline{\{Q_n\}} = \{Q_n,P_n\}, \overline{\{R_n\}} = \{R_n,P_n,S_n,T_n\}, \overline{\{S_n\}} = \{S_n\},\\
    \overline{\{T_n\}} = \{T_n\}, \overline{\{U_n\}} = \{U_n,T_n\}, \overline{\{V_n\}} = \{V_n,T_n,W_n,P_{n+1}\}, \overline{\{W_n\}} = \{W_n\}.
\end{multline*}
Let $f : \Pm(A)\to \mathbb{R}$ be a continuous function and suppose $f(P_1) = \cdots = f((P_n) = \alpha$. Then by the above equalities we must
have $\alpha = f(P_n) = f(Q_n) = f(R_n) = f(S_n) = f(T_n) = f(U_n) = f(V_n) = f(W_n) = f(P_{n+1})$ and we conclude that $f$ is a constant
function. We gather from the Dauns-Hofmann theorem that the center of the multiplier algebra of $A$ consists only of the scalar multiples of the
unit. On the other hand, $A$ has no unit by \cite[Proposition 2.13]{LT} and we are done.

\end{ex}

\bibliographystyle{amsplain}
\bibliography{}

% -----------------------------------
\end{document}